\documentclass[11pt]{amsart}

\usepackage{amsthm,amssymb,amsmath,calc,eucal,ifthen,srcltx}
\usepackage[matrix,arrow]{xy}

\newtheorem{theorem}{Theorem}[section]
\newtheorem{lemma}[theorem]{Lemma}

\theoremstyle{remark}

\newtheoremstyle{rmdefinition}{}{}{\upshape}{}{\bfseries}{.}{ }{}

\theoremstyle{rmdefinition}

\newcommand{\derived}[2][1]{\ifthenelse{\equal{#1}{1}}{{#2}^{\prime}}{\i 
fthenelse{\equal{#1}{2}}{{#2}^{\prime\prime}}{\ifthenelse{\equal{#1}{3}} 
{{#2}^{\prime\prime\prime}}{{#2}^{(#1)}}}}}

\begin{document}

\bibliographystyle{plain2}

\title[Correction to "Harmonic representatives \ldots "]{Correction to "Harmonic representatives for cuspidal cohomology classes"}

\author{J\'ozef Dodziuk }
\maketitle

\section{Introduction}
In this note we give a correction to the proof of the main result of \cite{Dod-McG-P}.
The proof of Lemma 4.1 in \cite{Dod-McG-P} was incorrect  in its use of Harnack inequality as pointed out by P.~Buser and E.~Makover. P.~Buser
suggested a workaround presented here. I am very grateful to Buser and Makover for
alerting me to the problem and for discussions that led to its resolution. 

We will follow the notation and conventions established in the first three sections
of \cite{Dod-McG-P}. This note should be regarded as a replacement for Section 4 of 
that paper.

\section{Harmonic differentials on a punctured disk and Laurent series}
There is a one-to-one correspondence between harmonic differentials on a coordinate
neighborhood $U$ on a Riemann surface and holomorphic functions of the complex local
coordinate where the differential $\omega$ corresponds to the function $f(z)$ and the
two are related by 
\begin{equation}\label{differential}
\omega = \text{Re}(f(z)dz). 
\end{equation}
Moreover, writing $z=x+iy$ we have 
\begin{equation} \label{holomorphic}
\Vert \omega \Vert_U^2 = \int_U \omega\wedge *\omega = \int_U (a^2+b^2)\,dxdy=
\int_U \vert f \vert^2 \,dxdy.
\end{equation}
Let $\omega$ be a harmonic differential on an annulus $A_{\rho,1}=\{ z = x+iy \mid
\rho^2 <x^2+y^2 < 1\}$ with period $p$ on the cycle surrounding the origin and
$f(z)$ be the corresponding holomorphic function. Suppose that $$
f(z) = \sum_{-\infty}^\infty a_n z^n$$
is the Laurent series expansion of $f(z)$ and that $z=\rho e^{i\theta}$. 

\begin{lemma}\label{formulae} The following formulae hold.
\begin{equation} \label{pdt}
p=-2\pi \operatorname{Im}(a_1)
\end{equation}
\begin{equation} \label{l2-norm}
\begin{split}
\Vert f\Vert^2_{A_{\rho,1}} &= \sum_{n=-\infty}^{\infty} \Vert a_n z^n\Vert^2_{A_{\rho,1}}\\
 &= 2\pi \sum_{n=2}^\infty|a_{-n}|^2 \,\frac{\rho^{2(1-n)}-1}{2(n-1)}\\
 &+ 2\pi\, |a_{-1}|^2\, \ln (1/\rho) \\
 &+ 2\pi \sum_{n=0}^\infty |a_n|^2 \,\frac{1-\rho^{2(n+1)}}{2(n+1)} 
\end{split}
\end{equation}
In addition, if $t=\theta/2\pi$, the holomorphic function corresponding to the harmonic differential
$pdt$ is $p/2\pi i z$.

\end{lemma}
\begin{proof} The proof is elementary. The Laurent series converges uniformly on 
compact subsets of the annulus so that integration can be performed term by term. 
(\ref{pdt}) follows from the fact that the period of $z^n\,dz$ is zero unless $n=-1$. (\ref{l2-norm}) is proved by a calculation using the fact that the powers $z^n$ form an orthogonal system
of functions. For the last comment observe that 
$$dt=\frac{1}{2\pi}d\theta=\operatorname{Re}\left (\frac{1}{2\pi i z} dz \right ).$$ 
\end{proof}
\section{Proof of Theorem 2.1 of \cite{Dod-McG-P}}
We parametrize all cusps of our Riemann surface $X$ as in \cite{Dod-McG-P} and cut
them off at $r=R$ for a large value of the parameter $R$ to obtain the surface $X_R$.
Since we allow our surface to have expanding ends (funnels), we consider also the
surfaces $\overline{X}$ and $\overline{X}_R$ obtained by adding circles at infinity
to funnels.
 Let $\omega_R$ be the unique harmonic differential on $\overline{X}_R$ satisfying the absolute boundary conditions along $\partial \overline{X}_R$ and with periods prescribed
by the cohomology class $c$ in the statement of Theorem 2.1 of \cite{Dod-McG-P}. By
Hodge theory (cf. \cite{ray-singer}) $\omega_R$ minimizes the $L^2$ norm in its cohomology class.
To use this observation effectively we will need a closed form of degree one on $\overline{X}$
which is equal to a constant multiple of $dt$ on every cusp.  Let $C$ be one of the
cusps and let $p$ be the period of the cohomology class $c$ along a horocycle on the
cusp. The form $\omega_R -p dt$ has period zero along horocycles and is therefore exact
on $C_{[-\ln 2, R]}$ so that $\omega_R -p dt= d\phi$ for a smooth function $\phi$ on $C_{[-\ln 2, R]}$. Let $g(t)$ be a smooth nonnegative function such that $g |_{(-\infty, -1/2]} \equiv 1$ and $g |_{(0, \infty)} \equiv 0$. Consider the form $pdt+d(g\phi)$. This form
is equal to $\omega_R$ near $r=-\ln 2$ and is equal to $pdt$ for $r>0$. Performing this construction on each cusp and extending the resulting form to be equal to
$\omega_R$ on the complement of the union of cusps we obtain a form $\alpha$ with
desired properties. We thus have
\begin{lemma}\label{alpha}
There exists a closed differential $\alpha$ on $\overline{X}$ with prescribed periods such
that on every cusp $C^i$ $\alpha$ is equal to $p_i dt$.
\end{lemma}

From now on it will be convenient to denote by $C$ the union of all cusps $C^i$.
Similarly, for an interval $I$, $C_I$ will stand for the union of subsets of all cusps
where the $r$ parameter belongs to the interval $I$. In addition, we will write 
$pdt$ for the form on $C$ whose restriction to the cusp $C^i$ is $p_i dt$. 

The following lemma will imply that the differentials
$\{\omega_R\}_{R\geq 0}$ form a "normal family." 

\begin{lemma}\label{uniform}
For every $a\geq 0$ there exists a constant $m(a)$ independent of $R$ such that
for all $R\geq a$ $$
\Vert \omega_R\Vert_{X_a} \leq m(a).$$.
\end{lemma}
\begin{proof}
We begin by observing that the form $dt$ is harmonic and satisfies absolute boundary conditions on $C_{[r_1,r_2]}$ for every $r_1 < r_2$. Therefore, as above, $pdt$ minimizes the norm it its cohomology class on $C_{[a,R]}$. It follows that
\begin{equation} \label{alpha-min}
 \Vert pdt \Vert_{C_{[a,R]}} \leq \Vert \omega_R\Vert_{C_{[a,R]}}
\end{equation}
On the other hand, $$
\Vert \omega_R\Vert_{X_R} \leq \Vert \alpha \Vert_{X_R}.$$
Therefore \begin{equation}\label{omega-alpha-pdt}
\Vert \omega_R\Vert_{X_a}^2 +  \Vert \omega_R\Vert_{C_{[a,R]}}^2 \leq
\Vert \alpha_R\Vert_{X_a}^2 +  \Vert \alpha\Vert_{C_{[a,R]}}^2 =
\Vert \alpha_R\Vert_{X_a}^2 +  \Vert pdt\Vert_{C_{[a,R]}}^2
\end{equation}
since $\alpha=pdt$ on cusps. The conclusion follows from (\ref{alpha-min}) if we set $m(a)=\Vert \alpha_R\Vert_{X_a}$.
\end{proof}
As explained in \cite{Dod-McG-P}, the lemma implies that there exists a sequence
$R_j$ approaching infinity such the the forms $\omega_{R_j}$ converge uniformly with
all derivatives on compact subsets of $\overline{X}$ to a form $\omega$ which is harmonic, has
prescribed periods and satisfies the absolute boundary conditions along boundaries of funnels. To finish the proof we need to investigate the
behavior of $\omega$ in cusps.

\begin{lemma}\label{omega-r-growth}
There exists a constant $m_0$ such that for every $0<r<R$
$$
\Vert \omega_R\Vert_{C_{[0,r]}}^2 \leq m_0 + \Vert pdt\Vert_{C_{[0,r]}}^2.
$$
\end{lemma}
\begin{proof}
We proceed as in (\ref{omega-alpha-pdt}).
$$
\Vert \omega_R\Vert_{X_0}^2 + \Vert \omega_R\Vert_{C_{[0,r]}}^2
+ \Vert \omega_R\Vert_{C_{[r,R]}}^2 \leq
\Vert \alpha\Vert_{X_0}^2 +  \Vert \alpha\Vert_{C_{[0,r]}}^2 +
\Vert \alpha\Vert_{C_{[r,R]}}^2
$$
We drop the first term on the left and, since $\alpha=pdt$ in cusps, rewrite the resulting inequality as 
\begin{equation}
\label{omega-r-bound}
\Vert \omega_R\Vert_{C_{[0,r]}}^2 \leq \Vert\alpha\Vert_{X_0}^2 + \Vert pdt\Vert_{C_{[0,r]}}^2  + \Vert pdt \Vert_{C_{[r,R]}}^2 - \Vert\omega_R\Vert_{C_{[r,R]}}^2
\end{equation}
which proves the conclusion in view of (\ref{alpha-min}).
\end{proof}
We now prove that $\omega -pdt$ is square-integrable in cusps. 
Passing to the limit as $R_j$ approaches infinity in (\ref{omega-r-bound}) we see that 
\begin{equation}\label{growth-omega}
\Vert \omega\Vert_{C_{[0,r]}}^2 \leq m_0 + \Vert pdt\Vert_{C_{[0,r]}}^2.
\end{equation}
To simplify the notation we continue the proof as if there was only one cusp. 
We represent the cusp $C_{[0,\infty)}$ as the punctured unit disk. Under this equivalence $C_{[0,r]}$ is mapped onto the annulus $A_{\rho,1}$ with $\rho=\rho(r)$
approaching zero as $r$ tends to infinity.  By Lemma \ref{formulae} $$
\Vert pdt\Vert_{C_{[0,r]}}^2 = 2\pi \vert \operatorname{Im}(a_{-1})\vert^2\ln (1/\rho ).$$
By (\ref{growth-omega}), the difference $\Vert f\Vert^2_{A_{\rho,1}} - 2\pi \vert \operatorname{Im}(a_{-1})\vert^2\ln (1/\rho )$ is bounded when $\rho$ approaches $0$.
In view of the formula (\ref{l2-norm}), this is possible only if $a_{-n}=0$ for $n\geq 2$ and $\operatorname{Re}(a_{-1}) =0$. Thus there are no negative powers of $z$
in the Laurent series expansion of the function corresponding to the differential
$\omega - pdt$ so that this function has a removable singularity at the origin. In particular this function is square-integrable on the punctured disk and therefore
$\omega -pdt$ is square-integrable on cusps.

The proof of existence of $\omega$ with required properties is complete. The uniqueness is proved as in Section 4 of \cite{Dod-McG-P}.
\bibliography{math}
\bibliographystyle{plain2}

\end{document}